\documentclass[runningheads, envcountsame, orivec]{llncs}
\usepackage{graphicx} % Required for inserting images
\usepackage[numbers,sort]{natbib}
\usepackage{amssymb,amsmath}
\usepackage{hyperref,cleveref}
\usepackage{enumerate}
\usepackage{mathtools}

\usepackage{etoolbox}
\AtEndEnvironment{proof}{\phantom{}\qed}

\usepackage{thmtools}
\usepackage{thm-restate}

\let\llncsproof\proof
\renewcommand{\proof}[1][Proof]{%
  \def\proofname{#1}%
  \llncsproof}

\usepackage{todonotes}

\setuptodonotes{inline}

\newcommand{\B}{\mathcal{B}}
\newcommand{\I}{\mathcal{I}}
\newcommand{\R}{\mathbb{R}}
\newcommand{\Z}{\mathbb{Z}}
\newcommand{\Rten}{R_{10}}
\newcommand{\Rtwelve}{R_{12}}
\newcommand{\onesum}{\oplus_1}
\newcommand{\twosum}{\oplus_2}
\newcommand{\threesum}{\oplus_\Delta}
\newcommand{\dual}[1]{{#1}^*}
\newcommand{\binary}{\mathbb{F}_2}
\newcommand{\bgp}[1]{f_{#1}}
\newcommand{\symdiff}{\triangle}
\newcommand{\ratfun}{\R(x)}

\DeclareMathOperator{\trop}{trop}

\crefname{theorem}{theorem}{theorems}
\Crefname{theorem}{Theorem}{Theorems}

\crefname{proposition}{proposition}{propositions}
\Crefname{proposition}{Proposition}{Propositions}

\crefname{corollary}{corollary}{corollaries}
\Crefname{corollary}{Corollary}{Corollaries}

\crefname{lemma}{lemma}{lemmas}
\Crefname{lemma}{Lemma}{Lemmas}

\crefname{definition}{definition}{definitions}
\Crefname{definition}{Definition}{Definitions}

\crefname{remark}{remark}{remarks}
\Crefname{remark}{Remark}{Remarks}

\newenvironment{acknowledgements}{\section*{Acknowledgements}}{}

\title{Arithmetic Circuits and Neural Networks for Regular Matroids\thanks{An extended abstract of this article appeared in the proceedings of the 27th Conference on Integer Programming and Combinatorial Optimization (IPCO~2026)~\cite{hertrich2026arithmetic}.}
}
\author{Christoph Hertrich\inst{1} \and Stefan Kober\inst{2} \and Georg Loho\inst{3}}
\authorrunning{C.\ Hertrich, S.\ Kober, and G.\ Loho}
\institute{University of Technology Nuremberg, Germany\\
\email{christoph.hertrich@utn.de}
\and
London School of Economics and Political Science, United Kingdom\\
\email{s.kober@lse.ac.uk}
\and
Freie Universität Berlin, Germany\\
\email{georg.loho@math.fu-berlin.de}
}
\date{\today}

\begin{document}

\maketitle

\begin{abstract}
We prove that there exist uniform $(+,\times,/)$-circuits of size $O(n^3)$ to compute the basis generating polynomial of regular matroids on $n$ elements. By tropicalization, this implies that there exist uniform $(\max,+,-)$-circuits and ReLU neural networks of the same size for weighted basis maximization of regular matroids. As a consequence in linear programming theory, we obtain a first example where taking the \emph{difference of two extended formulations} can be more efficient than the best known individual extended formulation for the base polytope of regular matroids of size $O(n^6)$ by Aprile and Fiorini. 
Such differences have recently been introduced as \emph{virtual extended formulations}. The proof of our main result relies on a fine-tuned version of Seymour's decomposition of regular matroids which allows us to identify and maintain graphic substructures to which we can apply a \emph{local} version of the \emph{star-mesh transformation}.
%\subclass{MSC code1 \and MSC code2 \and more}
\end{abstract}

\section{Introduction}
A key challenge in complexity theory is to characterize the computational power of \emph{arithmetic circuits}, the algebraic counterparts of Boolean circuits.
In such a circuit, every gate outputs an arithmetic expression, like the sum or product of its inputs.  The fundamental underlying motivation for this field of research is to understand which problems one can solve efficiently with a predefined set of algebraic operations. The most frequently studied type of arithmetic circuits is $(+,\times)$-circuits~\cite{shpilka2010arithmetic}, as they form a very natural model to represent polynomials. However, allowing additional operations, e.g., subtraction or division, sometimes makes the model of computation exponentially more powerful~\cite{valiant1979negation,jerrum1982some,hertrich2026tropical}. This leads to the question of how the complexity of representing certain functions depends on the operations allowed in the circuit. To this end, Fomin, Grigoriev, and Koshevoy~\cite{fomin2016subtraction} coined the notion of \emph{subtraction-free complexity}, which is the study of $(+,\times,/)$-circuits. These circuits are of particular interest because they can be \emph{tropicalized}: they give rise to corresponding $(\max,+,-)$-circuits computing \emph{tropical polynomials} or \emph{tropical rational functions}~\cite{jukna2023tropical}.

One motivation to study tropical circuits is to prove lower bounds on (pure) dynamic programming algorithms~\cite{jukna2015lower}. Our primary motivation, however, is the expressive power of feed-forward neural networks with ReLU activations, the most commonly used activation function in modern machine learning. Like $(\max,+,-)$-circuits, ReLU networks represent \emph{continuous and piecewise linear} (CPWL) functions~\cite{arora2018understanding}, which can be understood as tropical rational functions. The tropical viewpoint on neural networks was initiated by Zhang, Naitzat, and Lim~\cite{zhang2018tropical} and Charisopoulos and Maragos~\cite{charisopoulos2018tropical}, and was subsequently used in many theoretical works on neural networks, e.g., to prove lower bounds on the depth~\cite{hertrich2023towards,haase2023lower}.

We contribute to an emerging stream of research studying ReLU networks as a kind of arithmetic circuit~\cite{hertrich2024relu,bakaev2025depth,hertrich2026tropical}. 
In fact, ReLU networks can be interpreted as $(\max,+,-)$-circuits with the additional power of multiplying intermediate results with real-valued constants (through the \emph{weights} of the neural network). Thereby, using the idea of tropicalizing $(+,\times,/)$-circuits, this yields a recipe to transform subtraction-free circuits into ReLU networks.

\paragraph{Contribution.} Our main result is the construction of efficient $(+,\times,/)$-circuits for the basis generating polynomial of regular matroids. For a matroid~$M$ with basis set $\B$, the \emph{basis generating polynomial} is $\bgp{M}\coloneqq\sum_{B\in\B}x^B\coloneqq\sum_{B\in\B}\prod_{e\in B} x_e$.

\begin{restatable}{theorem}{mainthm}\label[theorem]{main-thm}
	For a regular matroid $M$ with $n$ elements, there is a $(+,\times,/)$-circuit of size $O(n^3)$ computing the basis generating polynomial $\bgp{M}$.
	Given an independence oracle of $M$, this circuit can be constructed in polynomial time.
\end{restatable}

\Cref{main-thm} generalizes a result by Fomin et al.~\cite{fomin2016subtraction} on subtraction-free circuits for the spanning tree polynomial. In circuit complexity, a circuit family whose members can be computed in polynomial time is called \emph{uniform}.

\paragraph{Implications on neural networks.}

Using the idea of tropicalization, which we make formal in \Cref{prop:circuit-tropicalization,prop:circuit-to-nn}, we obtain the following immediate corollary.

\begin{corollary}\label[corollary]{cor:tropical-main}
	For a regular matroid $M$ with basis set $\B$ and $n$ elements, there is (i) a $(\max,+,-)$-circuit and (ii) a ReLU neural network of size $O(n^3)$ computing the tropical polynomial $\max_{B\in\B}\sum_{e\in B} x_e$. Given an independence oracle of $M$, the circuit and the neural network can be constructed in polynomial time.
\end{corollary}

Evaluating this tropical polynomial can be seen as solving the \emph{maximum weight basis problem} over the regular matroid, which can famously be achieved by the greedy algorithm. However, as described by Hertrich and Sering~\cite{hertrich2024relu}, finding efficient neural network representations to solve combinatorial optimization problems like this remains non-trivial, as the computational model of neural networks is missing simple algorithmic building blocks like if-branchings based on the comparison of real numbers or computing argmins over finite sets. These limitations prohibit the approach of simply implementing the greedy algorithm on a ReLU neural network. 
Consequently, \Cref{cor:tropical-main} requires a deeper understanding of the structure of regular matroids and provides new non-trivial insights on the power of ReLU neural networks as a model of computation.

\paragraph{Implications in linear programming theory.}

Besides implications in neural network theory, our main result also implies a new upper bound on the so-called \emph{virtual extension complexity} of regular matroid base polytopes. For a given polytope~$P$, an \emph{extended formulation} of $P$ is a system of linear inequalities describing a higher-dimensional polytope that projects onto $P$. Its \emph{extension complexity} $\mathrm{xc}(P)$ is the minimum number of inequalities in such a system, thereby measuring the smallest possible size of a \emph{single} linear program for optimizing over $P$.

A generalization of this concept is the \emph{virtual extension complexity}~\cite{hertrich2024neural}, defined as
\[
\mathrm{vxc}(P)\coloneqq\min\{\mathrm{xc}(Q)+\mathrm{xc}(R):Q,R\text{ polytopes},\ P+Q=R\},
\]
where $P+Q$ denotes the Minkowski sum of $P$ and $Q$. Given small extended formulations for $Q$ and $R$, one can find an optimal solution over $P$ in polynomial time by solving linear programs over $Q$ and $R$~\cite[Theorem~4.6]{hertrich2024neural}. In that sense, $\mathrm{vxc}(P)$ measures the minimum number of inequalities necessary to represent the optimization problem over $P$ as the \emph{difference of two linear programs}. It is an open question whether $\mathrm{vxc}(P)$ can be smaller than $\mathrm{xc}(P)$, that is, whether taking a single difference of two LPs brings an advantage over solving just one LP.

In fact, every ReLU network that exactly represents the support function of a polytope $P$ can be turned into a virtual extended formulation for $P$ of asymptotically the same size as the neural network~\cite[Theorem~4.1]{hertrich2024neural}. As the tropical polynomial in \Cref{cor:tropical-main} is precisely the support function of the base polytope of $M$, we immediately obtain the following corollary.

\begin{corollary}
	There exists a \emph{virtual} extended formulation of size $O(n^3)$ of the matroid base polytope of a regular matroid on $n$ elements.
\end{corollary}

This is in contrast to the best known extension complexity upper bound $O(n^6)$ from~\cite{aprile2022regular}.
Thereby, to the best of our knowledge, our result provides the first example of virtual extended formulations that are smaller than the best known extended formulations.

\paragraph{Generalization to MFMC matroids.}
Our results extend to the more general class of Max-Flow-Min-Cut (MFMC) matroids. 
This class of matroids has been introduced by Seymour~\cite{seymour1977matroids}, and MFMC matroids have been shown to have a decomposition that extends Seymour's decomposition of regular matroids~\cite{seymour1980decomposition}, see~\cite{seymour1977matroids} and~\cite[Corollary~12.3.22]{oxley2006matroid}.
We formally define MFMC matroids and explain how our main result can be generalized in \Cref{sec:mfmc}.

\paragraph{Proof techniques.}
A key ingredient to prove our main result is Maurer's generalization~\cite{maurer1976matrix} of Kirchhoff's celebrated matrix-tree theorem~\cite{kirchhoff1847ueber} in order to view the basis generating polynomial of a regular matroid $M$ as the determinant of an $r$-dimen\-sional square matrix~$L$, where~$r$ is the rank of~$M$.
Since a general determinant cannot be computed with a $(+,\times,/)$-circuit, we need to exploit the specific structure of the matrix $L$.
We can understand the approaches by Fomin et al.~\cite{fomin2016subtraction} and Hertrich and Sering~\cite{hertrich2024relu} for the case of spanning trees in graphs as an inductive approach to reduce the number of rows of $L$ by one in each step.
In terms of graphs, this strategy corresponds to eliminating one vertex at a time, and compensating for the lost information by introducing additional edges between neighbors of the removed vertex. This procedure is known as a \emph{star-mesh transformation}~\cite{fomin2016subtraction}.
In general regular matroids, however, the concept of a vertex does not exist, making it difficult to generalize this approach. Nevertheless, we introduce a generalization of the star-mesh transformation to $\{0,\pm1\}$-matrices, see \Cref{gen-star-mesh}.
However, this generalization suffers from the fact that it does not preserve regularity, rendering the application of Maurer's theorem invalid in later iterations. In fact, for the regular matroid $\Rten$, none of the possible generalized star-mesh transformations preserves regularity.
We show that this is the only counterexample, in the following sense. We give a refined version of Seymour's decomposition of regular matroids, implying that for any $3$-connected regular matroid $M$ that is not $\Rten$, either $M$ or its dual has a star-mesh transformation that preserves regularity.
Together with constructions of arithmetic circuits for the $1$-sum and $2$-sum of regular matroids, $\Rten$, and all cographic matroids by dualization, this allows an inductive proof of our main result.
We follow a similar strategy as B{\'e}rczi, M{\'a}trav{\"o}lgyi, and Schwarcz~\cite{berczi2024reconfiguration}, and employ Seymour's decomposition of regular matroids without explicitly dealing with the most complicated case, that is, the $3$-sum. 
Instead, we are able to iteratively reduce the rank by~$1$ in any $3$-connected regular matroid, see \Cref{prop:removevertex}. 
It would be interesting to see whether this strategy can be employed to other problems on regular matroids, for which in many cases the $3$-sum is the hardest case~\cite{artmann2017strongly,artmann2020optimization,aprile2022regular} or remains unsolved, see e.g.,~\cite{aprile2025integer}, and extensions of~\cite{nagele2025advances,berczi2024reconfiguration}.

\paragraph{Limitations beyond regular matroids.}

We do not expect that \Cref{main-thm} can be generalized much beyond regular or MFMC matroids. The reason is that efficient circuits for the basis generating polynomial imply a simple algorithm to count the number of bases: simply feed the all-ones vector into the circuit. However, already for binary matroids, which would be the natural next step to generalize \Cref{main-thm}, counting the bases is $\#\textsf{P}$-hard~\cite{knapp2025complexity,snook2012counting}, see also~\cite{anari2018log,anari2019log} for some approximate counting results. Note that this limitation is not specific to the subtraction-free setting, as it equally holds for circuits allowing subtractions.
On a more technical level, our proofs crucially rely on the generalization of the Matrix Tree Theorem to regular matroids~\cite{maurer1976matrix}. Such a generalization is not known beyond regular matroids. 
Unlike \Cref{main-thm}, it may be possible that \Cref{cor:tropical-main} could be generalized to binary matroids, but this would require entirely different proof techniques that are specific to the tropical setting.
A related open problem is whether base polytopes of binary matroids always have polynomial extension complexity~\cite{aprile2022regular}.
A first step in this direction could be to consider proper minor-closed subclasses of binary matroids, which admit strong structural results~\cite{geelen2015highly}.  

\paragraph{Related work.}
The notion of subtraction-free circuits was coined by Fomin et al.~\cite{fomin2016subtraction}, who also proved that $(+,\times,/)$-circuits of size $O(n^3)$ exist for the spanning tree generating polynomial of a graph with $n$ vertices. This is the special case of \Cref{main-thm} for graphic matroids. Hertrich and Sering~\cite{hertrich2024relu} implicitly used this to conclude that ReLU networks of size $O(n^3)$ can compute the value of a minimum spanning tree from the edge weights. They also constructed polynomial-size ReLU networks for the maximum flow problem. ReLU networks as a model of computation were also studied in \cite{hertrich2023provably} for the knapsack problem.
The book by Jukna~\cite{jukna2023tropical} provides an in-depth treatment of tropical circuits. In particular, Jukna and Seiwert~\cite{jukna2019greedy} proved an exponential lower bound for $(+,\times)$-circuits computing the spanning tree generating polynomial. This implies that, already in the graphic case, \Cref{main-thm} would fail without division gates. This lower bound also translates to the tropical setting and therefore shows that subtraction gates are necessary in \Cref{cor:tropical-main}.
Hertrich and Stargalla~\cite{hertrich2026tropical} established similar lower bounds for tropical circuits augmented with scalar multiplication gates.
When disallowing subtractions in ReLU networks, the resulting model of \emph{monotone} or the related \emph{input-convex} neural networks yield a different model of computation that was investigated in~\cite{hertrich2024neural,bakaev2025depth,brandenburg2024decomposition,gagneux2025convexity,hertrich2026tropical}. Such networks imply the existence of (ordinary) extended formulations for the polytopes underlying the related optimization problem~\cite{hertrich2024neural}, while non-monotone networks only yield virtual extended formulations. Connections between circuit complexity and extension complexity also appeared in~\cite{hrubevs2023shadows,fiorini2021strengthening}. It remains an open question whether there is any class of CPWL functions that can be evaluated in polynomial time, but requires exponential-size ReLU networks.

Arora, Basu, Mianjy, and Mukherjee~\cite{arora2018understanding} proved that the class of functions computable by ReLU networks exactly coincides with the class of CPWL functions.
We refer to~\cite{huchette2023deep} for a survey on polyhedral methods in deep learning.
It is an open question which depth one needs for such exact representations~\cite{hertrich2023towards,haase2023lower,averkov2025expressiveness,bakaev2025better,valerdi2024minimal,grillo2025depth,rueß2026shallowerrelunetworkrepresentations}. This is in contrast to well-known universal approximation theorems~\cite{cybenko1989approximation,leshno1993multilayer}, which are often restricted to a bounded domain and require very wide neural networks.

Our approach for regular matroids is based on a refinement of Seymour's decomposition of regular matroids~\cite{seymour1980decomposition}. 
Such refinements have been crucial in many other algorithmic and structural applications of this decomposition theorem, see, e.g.,~\cite{dinitz2014matroid,aprile2022regular,berczi2024reconfiguration}.
We emphasize that our refinement naturally extends text-book proofs of Seymour's decomposition and is otherwise self-contained. 

\section{Preliminaries}
\label{sec:prelims}

This section presents the necessary preliminaries to understand our main results and proof techniques.

\subsection{Arithmetic circuits and neural networks}
An \emph{arithmetic circuit} is a directed acyclic graph defining an arithmetic expression using 2-ary operations like $\max$, $\min$, $+$, $-$, $\times$, or $/$.
We assume that each node (or \emph{gate}) of the circuit has either in-degree zero, in which case it is one of $n$ input nodes, or in-degree two, in which case it performs one of the previously mentioned arithmetic computations on the outputs of its two predecessors. Further, we assume that there is a unique output gate with out-degree zero, which defines the arithmetic expression represented by the entire circuit.
This expression contains one formal variable for each of the $n$ input gates.
The \emph{size} of an arithmetic circuit is the number of non-input gates.
When talking about specific types of arithmetic circuits, we usually put the allowed operations in parentheses; e.g., a $(+,\times,/)$-circuit is an arithmetic circuit where, besides input gates, there only exist addition, multiplication, and division gates. In the context of $(+,\times,/)$-circuits, we will often exploit that every gate represents a subtraction-free rational function, that is, a quotient of nonzero polynomials with nonnegative coefficients. In this setting, no division gate divides by a function that is identically zero. Thus, division by zero does not occur at the level of formal rational functions.

In the context of this paper, we consider \emph{neural networks} with \emph{rectified linear unit} (ReLU) activations. 
Like arithmetic circuits, they define a computation through a directed acyclic graph with $n$ input nodes (or \emph{neurons}) and a particular output neuron. In contrast to arithmetic circuits, the neurons can have arbitrary large in-degree (say $k$) and compute a function $(z_1\dots,z_k)\mapsto \max\{0,\sum_{i=1}^k w_i z_i\}$, where the $z_i$ are the outputs of the predecessor neurons and the $w_i$ are the \emph{weights} of the considered neuron. Then, the whole network computes a \emph{continuous and piecewise linear} (CPWL) function $f\colon\R^n\to\R$ defined on the input variables and parameterized by the weights of all the neurons. The \emph{size} of such a neural network is the number of non-input neurons.

We will now make some statements to relate arithmetic circuits, their tropical counterparts, and neural networks. To this end, we define the \emph{tropicalization} of a $(+,\times,/)$-circuit as the corresponding $(\max,+,-)$-circuit, where every $+$-gate is replaced with a $\max$-gate, every $\times$-gate with a $+$-gate, and every $/$-gate with a $-$-gate. 
The function computed by a $(+,\times,/)$-circuit is a \emph{rational function}
\begin{equation}\label{eq:rational_function}
	f(\mathbf z)=\frac{\sum_{\mathbf i \in \mathcal I_1} a_{\mathbf i} \mathbf z^{\mathbf i}}{\sum_{\mathbf i \in \mathcal I_2} b_{\mathbf i} \mathbf z^{\mathbf i}},
\end{equation}
that is, the quotient of two polynomials, where $\mathbf i=(i_1,\dots,i_n)\in\Z_{\geq0}^n$ is a multi-index and $a_{\mathbf i},b_{\mathbf i}\in \Z_{\geq1}$. 
We highlight that by our definition of a $(+,\times,/)$-circuit, the constant coefficients $a_{\mathbf i},b_{\mathbf i}\in \Z_{\geq1}$ only arise from accumulated additions. 
We do not allow the multiplication with arbitrary constant values.
In particular, the process of tropicalization would not yield the desired construction if we had negative coefficients.

For a rational function of the form \eqref{eq:rational_function}, we define its \emph{tropicalization} as the CPWL function
\begin{equation}\label{eq:tropical_function}
	\trop(f)(\mathbf z)=\max_{\mathbf i \in \mathcal I_1} \mathbf i^\top\mathbf z-\max_{\mathbf i \in \mathcal I_2} \mathbf i^\top\mathbf z.
\end{equation}
See, e.g., Joswig~\cite[Section~2.2]{joswig2021essentials} for more details on the process of tropicalization. 
Observe that our version of tropicalization does not depend on the values of $a_{\mathbf i}$ and $b_{\mathbf i}$, as long as they are non-zero. This is on purpose to induce the following behavior: $\trop(z) = z = \max\{z,z\} = \trop(z + z) = \trop(2z) $. 
In other words, one might think of this as all non-zero constants being sent to $0$ by the map $\trop$.

With these definitions we can derive the following two propositions, which immediately imply that \Cref{cor:tropical-main} can be deduced from \Cref{main-thm}. See also \cite[Corollary~6.4]{jukna2023tropical} and \cite{hertrich2024relu} for specific versions in the case of spanning trees.

\begin{restatable}{proposition}{circuittropicalization}
	\label[proposition]{prop:circuit-tropicalization}
	If a $(+,\times,/)$-circuit computes a rational function $f$, then the tropicalization of this circuit computes $\trop(f)$.
\end{restatable}

To prove \Cref{prop:circuit-tropicalization}, we use the following lemma.

\begin{lemma}\label[lemma]{lem:semiringhomo}\
	\begin{enumerate}[(i)]
		\item The rational functions of the form \eqref{eq:rational_function} together with $+$ and $\times$ form a semiring with multiplicative inverses.
		\item The CPWL functions of the form \eqref{eq:tropical_function} together with $\max$ and $+$ form a semiring with multiplicative inverses.
		\item The map $\trop$ is a semiring homomorphism between these two semirings.
	\end{enumerate}
\end{lemma}
Note that, in particular, neither the class of rational functions nor CPWL functions have in general additive inverses in these respective semirings.
\begin{proof}
	Items (i) and (ii) are straight-forward to verify. To prove item (iii), first note that the same statement is known for the case of polynomials instead of rational functions, compare \cite[Prop.~2.8]{joswig2021essentials}\footnote{Note that no cancellation can occur when adding two polynomials as we assume the coefficients to be positive.}. To also see it for rational functions, let $f/g$ and $f'/g'$ be two rational functions of the form \eqref{eq:rational_function}. Using the semiring homomorphism property for polynomials, we calculate 
	\begin{align*}
		&\phantom{{}={}}\trop\Big(\frac{f}{g} + \frac{f'}{g'}\Big)\\ &= \trop\Big(\frac{fg'+f'g}{gg'}\Big) = \trop(fg'+f'g)-\trop(gg') \\&= \max\{\trop(f)+\trop(g'),\trop(f')+\trop(g)\}-\trop(g)-\trop(g')\\&=
		\max\{\trop(f)-\trop(g),\trop(f')-\trop(g')\}
		\\&= \max\Big\{\trop\Big(\frac{f}{g}\Big),\trop\Big(\frac{f'}{g'}\Big)\Big\}
	\end{align*}
	and
	\begin{align*}
		\trop\Big(\frac{f}{g} \cdot \frac{f'}{g'}\Big) &= \trop\Big(\frac{ff'}{gg'}\Big) = \trop(ff')-\trop(gg') \\&= \trop(f)-\trop(g)+\trop(f')-\trop(g')\\&= \trop\Big(\frac{f}{g}\Big)+\trop\Big(\frac{f'}{g'}\Big),
	\end{align*}
	implying the statement.
\end{proof}

\begin{proof}[Proof of \Cref{prop:circuit-tropicalization}]
	We prove the claim by induction on the size of the circuit. Indeed, for the induction start, we use that the input gates compute the rational function $\mathbf z\mapsto z_j$ for some index $j$. 
	The tropicalization of this is still $z_j$, which is also what is computed by the corresponding input gate of a tropical circuit. For the induction step we consider the output gate of the circuit. By induction, the claim is true for the two sub-circuits that compute the expressions fed into the output gate. The claim then follows for the entire circuit using \Cref{lem:semiringhomo}.
\end{proof}

Note that the semiring homomorphism $\trop$ is not injective. Therefore, the converse of \Cref{prop:circuit-tropicalization} fails, that is, converting $(\max,+,-)$-circuits to $(+,\times,/)$.

\begin{proposition}\label[proposition]{prop:circuit-to-nn}
	If a CPWL function is computed by a $(\max,+,-)$-circuit of size $s$, then it can also be computed by a ReLU network of size $3s$.
\end{proposition}
\begin{proof}
	The circuit directly translates to a neural network, where each max-gate can be realized with three ReLU neurons using $\max\{x,y\} = \max\{0,x-y\} + y = \max\{0,x-y\} + \max\{0,y\} - \max\{0,-y\}$, see~\cite{arora2018understanding,hertrich2023towards}.
\end{proof}

In contrast to our circuit models, neural networks do have constants encoded in their weights. Networks arising through \Cref{prop:circuit-to-nn} have weights in $\{0,\pm1\}$. This is fine for proving upper bounds, but it remains unclear to what extent the use of different constants could lead to more efficient representations, see also~\cite{hertrich2026tropical}.

With \Cref{prop:circuit-tropicalization,prop:circuit-to-nn} at hand, \Cref{cor:tropical-main} is a direct consequence of \Cref{main-thm}.

\subsection{Matroid basics}\label{sec:matroids}
We assume the reader to be familiar with standard graph and matroid terms. 
For a comprehensive overview of regular matroids, we refer to~\cite{oxley2006matroid,truemper1992matroid}.
We lay out some matroid definitions in terms of the bases of matroids, where the definitions are necessary to follow our arguments.

\paragraph{Graph basics.}
We consider undirected graphs $G=(V,E)$, where $V$ is the set of vertices and $E$ is the set of edges.
Unless otherwise specified our graphs are simple, i.e., they do not have loops or parallel edges.
We denote the complete graph on $i$ vertices by $K_i$ and the complete bipartite graph with partition into stable sets of size $i$ and $j$ by $K_{i,j}$.
Let $G,H$ be graphs and $\{V_h:h\in V_h\}$ be a partition of $V(G)$ into connected subsets. 
We say that $G$ has an \emph{$H$-model}, if for every edge $h_1h_2\in E(H)$, there is an edge $v_1v_2\in E(G)$, such that $v_i\in V_{h_i}$ for $i\in[2]$.
We also say that $H$ is a \emph{minor} of $G$.
Given a subset of the vertices $X\subseteq V$ of $G$, let $E':=\{e\in E:e\subseteq X\}$.
We denote the \emph{induced subgraph} on $X$ by $G[X]:=(X,E')$.

We say that a graph $G=(V,E)$ is \emph{$k$-connected}, if its corresponding graphic matroid $M(G)$ is $k$-connected.
Observe that a $k$-separation for $G$ corresponds to a partition of the edge set $E=E_1\cup E_2$, where $|E_1|\ge k$, $|E_2|\ge k$ and $\left|\left(\bigcup_{e\in E_1}e\right) \cap \left(\bigcup_{e\in E_2}e\right)\right|\le k$.
Note that our definition of separations is closely related to vertex cuts, but does not coincide with the usual notion of edge or vertex cuts in graphs. 

\paragraph{Matroid basics.}
Given a ground set of elements $E$, and a non-empty set of bases $\B\subseteq 2^E$ we say that $M=(E,\B)$ is a \emph{matroid} if for each $B_1,B_2\in\B$ with $B_1\neq B_2$ and every $b\in B_1\setminus B_2$, there is a $b'\in B_2\setminus B_1$ such that $B_1\cup\{b'\}\setminus\{b\}\in\B$.
We set $n:=|E|$ and define the rank as $r:=r(M):=|B|$ for any $B\in\B$.
The rank function extends to $X\subseteq E$, i.e., $r(X):=\max\{|B\cap X|:B\in\B\}$.
We say that a set of elements $I\subseteq E$ is \emph{independent}, if it is a subset of a basis, and denote the set of independent sets by $\I(M)$.
Further, we say that $C\subseteq E$ is a \emph{circuit} of $M$, if $C\notin\I(M)$, but for all $e\in C$, $C-e\in\I(M)$.
Note that a matroid is uniquely determined by its bases, its independent sets, or its circuits.

Given a matroid $M=(E,\B)$, we define the \emph{cobases} as the set of complements of bases $\dual\B:=\{E\setminus B:B\in\B\}$.
Then the dual matroid $\dual M$ of $M$ is defined on the same ground set $E$, with set of bases $\dual\B$.
This directly implies $\dual{(\dual M)}=M$.
We say that $C\subseteq E$ is a cocircuit of $M$ if $C$ is a circuit of $\dual M$.

Given a matroid $M=(E,\B)$ and some element $e\subseteq E$, we define the deletion operation $M\setminus e$ as the matroid on the ground set $E\setminus \{e\}$ with bases $\{B\in\B:e\notin B\}$.
Similarly, we define the contraction operation $M/e$ as the matroid on the ground set $E\setminus\{e\}$ with bases $\{B\setminus\{e\}:B\in\B, e\in B\}$.
We say that a matroid $N$ is a minor of $M$, if we can obtain a matroid isomorphic to~$N$ from~$M$ by deleting a subset of elements~$U$ and contracting a subset of elements~$W$.
We say that a minor is proper if $|U|+|W|\ge1$. 

\paragraph{Representability.}
Given a field $\mathbb{F}$, we say that a matroid $M=(E,\B)$ of rank $r$ is \emph{$\mathbb{F}$-representable}, if there is a matrix $A\in\mathbb{F}^{r\times n}$, 
such that a subset $E'\subseteq E$ is independent if and only if the corresponding columns are linearly independent over $\mathbb{F}$. 
Note that elementary row operations do not affect the independence of columns.
Therefore, we may assume that any representation matrix of a matroid contains a full-rank identity matrix.
Given a representation matrix of the form $A=(I_r\,C)$ of $M$, where $I_r$ is the identity matrix of size $r$, a representation matrix of $\dual M$ is given by $(C^\intercal\,I_{n-r})$.

\paragraph{Regular matroids.}
Recall that a matrix is \emph{totally unimodular} (TU) if all its subdeterminants are in $\{-1, 0,1\}$ (over $\R$). 
Let $M$ be a binary matroid with representation matrix $A$ over $\binary$. 
We say that $M$ is \emph{regular} if the $1$-entries of $A$ can be signed such that we obtain a real totally unimodular matrix. 
This is equivalent to $M$ being representable over $\R$ by a TU matrix. 

We call a matroid $M=(E,\B)$ \emph{graphic} if there is a connected graph $G=(V,E)$, such that $C\subseteq E$ is a circuit of $M$ if and only if the corresponding edges form a cycle in $G$.
We write $M:=M(G)$ and say that $M(G)$ is the graphic matroid of $G$.
The bases of $M(G)$ map bijectively to the spanning trees of the corresponding graph $G$.
A matroid $M$ is said to be \emph{cographic} if $\dual M$ is graphic.
Graphic and cographic matroids are important examples of regular matroids, but not every regular matroid is of this form, see \Cref{thm:decomp}.

\paragraph{Matroid representations from oracles.}
The set of bases, independent sets, or circuits of a matroid can be exponentially large in terms of the number of elements, even for regular matroids.
Hence, matroids are often given implicitly in algorithmic applications, via oracles.
The most commonly used type of oracle is the \emph{independence oracle}. 
Let $M=(E,\B)$ be a matroid.
Then, an independence oracle takes a set $X\subseteq E$ as its input and returns \emph{`Yes'} if $X$ is independent, and \emph{`No'} otherwise.
In our running time statements throughout the paper, we assume that matroids are given via independence oracles and that one oracle call takes time $O(1)$.
A polynomial number of (independence) oracle calls suffices to obtain a binary representation matrix~$A$ in polynomial time, see e.g.,~\cite[Proposition~9.2.2, Proposition~9.4.23]{oxley2006matroid}.
If $M$ is regular, we can use Camion's algorithm~\cite{camion1964matrices} (see \cite[Corollary~9.2.7]{truemper1992matroid}) to find a signing of the entries of $A$ over $\R$ in order to find a totally unimodular representation matrix of $M$ in polynomial time.
Thus, we can always assume access to a representation matrix for any binary matroid, and to a totally unimodular signing if the matroid is regular, at the cost of a polynomial overhead in terms of elementary operations, including independence queries.
We remark that for general matroids, an independence oracle alone is not sufficient to test efficiently whether a given matroid is binary~\cite{seymour1981recognizing}.

Within our algorithm, we will need access to a specific representation of our matroid based on a given cocircuit.
The proof follows the same idea as in~\cite[Proposition~9.2.2, Proposition~9.4.23]{oxley2006matroid}.
We add a short proof for the convenience of the reader.
\begin{restatable}{lemma}{fundcocircuit}\label[lemma]{fund_cocircuit}
	Let $M$ be an $\binary$-representable matroid given via an independence oracle, and let $D\subseteq E(M)$ be a cocircuit of $M$.
	Then we can find a representation matrix $A\in\binary^{r\times n}$ of $M$ such that $D=\mathrm{supp}(A_r)$ in polynomial time, where $A_r$ denotes the $r$-th row of $A$. 
\end{restatable}
\begin{proof}
	Let $d\in D$, and consider a basis $\dual B$ of $\dual M$ such that $D\setminus\{d\}\subseteq \dual B$.
	Let $A^*=(I_{n-r}\ C)\in\binary^{(n-r)\times n}$ be the representation matrix of $\dual M$ where the columns of $C$ correspond to the incidence vectors of the \emph{fundamental circuits} with respect to $\dual B$, i.e., for every element $e\in E\setminus \dual B$ we take the unique circuit of $\dual M$ contained in $\dual B\cup \{e\}$.
	We permute the columns of $C$, such that the element $d$ corresponds to the last column of $A^*$. 
	Note that by definition, this column has a $1$ in a given row if and only if the corresponding column of the identity matrix corresponds to some element of $D\setminus\{d\}$.
	
	Clearly, $A:=(C^\intercal\ I)$ is a representation matrix of $M$, such that $D=\mathrm{supp}(A_r)$.
	Further, $B:=E(M)\setminus\dual B$ is a basis of $M$ such that $B\cap D=\{d\}$.
	Given any representation matrix $A'$ of $M$, we can efficiently find such a basis $B$, and perform elementary row operations such that the columns corresponding to $B$ are standard unit vectors.
	By the above arguments, the resulting matrix $A$ has the desired properties.
\end{proof}

\paragraph{Connectivity.}
Let $M$ be a matroid with ground set $E$.
We say that a partition $(X,E\setminus X)$ is a $k$-separation of $M$ if $r(X)+r(E\setminus X)-r(M)\le k-1$.
For $k\ge2$, the matroid $M$ is called \emph{(Tutte) $k$-connected}, if there is no $\ell$-separation with $\ell<k$.
Note that $\dual M$ is $k$-connected if and only if $M$ is $k$-connected. 

\paragraph{$1$-, $2$-, and $\Delta$-sums.}
Let $M_1=(E_1,\B_1)$ and $M_2=(E_2,\B_2)$ be matroids.
If $E_1\cap E_2=\emptyset$, then we define the \emph{$1$-sum of $M_1$ and $M_2$} as the matroid $M$ with ground set $E:=E_1\cup E_2$ and set of bases $\B:=\{B_1\cup B_2:B_1\in\B_1,B_2\in\B_2\}$.
We write $M:=M_1\onesum M_2$.
If $E_1\cap E_2=\{e\}$, where $e$ is neither a loop nor a co-loop in $M_1$ and $M_2$, then we define the \emph{$2$-sum of $M_1$ and $M_2$ at $e$} as the matroid with ground set $E:=E_1\symdiff E_2$ and set of bases $\B:=\{(B_1\cup B_2)-e:B_1\in\B_1,B_2\in\B_2,e\in B_1\symdiff B_2\}$.
We write $M:=M_1\twosum M_2$.

Let $M_1$ and $M_2$ be binary matroids, and $E_1\cap E_2=\{d_1,d_2,d_3\}:=D$, such that $D$ is a circuit and does not contain a cocircuit in both $M_1$ and $M_2$.
Then we define the \emph{$\Delta$-sum of $M_1$ and $M_2$ at $D$} as the matroid with ground set $E:=E_1\symdiff E_2$ and set of bases $\B$, where $B=(B_1\cup B_2)\setminus D\in\B$ if $B_1\in\B_1$, $B_2\in\B_2$, $B_1\cap B_2=\emptyset$ and either 
\begin{enumerate}[(i)]
	\item $|B_1\cap D|=0$ and $|B_2\cap D|=2$, \emph{or}
	\item $|B_1\cap D|=2$ and $|B_2\cap D|=0$, \emph{or}
	\item $B_1\cap D=\{d_i\}$, $B_2\cap D=\{d_j\}$, $B_1\symdiff(D-d_j)\in\B_1$, and $B_2\symdiff(D-d_i)\in\B_2$.\label{Delta3}
\end{enumerate}
We write $M:=M_1\threesum M_2$.
We remark that while our definition of the $1$-, $2$-, and $\Delta$-sum is non-standard, it is well-known to be equivalent to other definitions, see~\cite[Lemma~10]{aprile2022regular} for the most general case (the $\Delta$-sum) and~\cite{berczi2024reconfiguration}.
Further, note that our definition of $\Delta$-sum (following Truemper~\cite{truemper1992matroid}) corresponds to the way that the $3$-sum is defined in other contexts~\cite{aprile2022regular,oxley2006matroid,seymour1980decomposition}.
Finally, note that if $M:=M_1\onesum M_2$, then $(E(M_1),E(M_2))$ is a $1$-separation of $M$.
Analogous statements hold for the $2$-sum and the $\Delta$-sum, provided the summands contain sufficiently many elements.

\paragraph{$\Delta$-$Y$-exchanges.}
Let $M=(E,\B)$ be a matroid with three specified elements $D:=\{d_1,d_2,d_3\}\subseteq E$, such that $D$ forms a circuit and does not contain a cocircuit.
We define a new matroid $M'=(E,\B')$. 
Specifically, we say that $B'\in\B'$ if and only if
\begin{enumerate}[(i)]
	\item $B'\cap D=\{d_i\}$ and $B'-d_i\in\B$ for some $i\in[3]$, \emph{or}
	\item $|B'\cap D|=2$ and $B'\symdiff D\in\B$, ($\symdiff$ denotes the symmetric difference), \emph{or}
	\item $|B'\cap D|=3$ and $B'-d_i\in\B$ for some $i\in[3]$.
\end{enumerate}
We denote $Y_D(M):=M'$, and call the operation to obtain $M'$ from $M$ a \emph{$\Delta$-$Y$-exchange at $D$}.
Conversely, given three specified elements $D:=\{d_1,d_2,d_3\}\subseteq E$ such that $D$ forms a cocircuit and does not contain a circuit, we denote $\Delta_{D}(M):=\dual{(Y_{D}(\dual M))}$ and call the corresponding operation a \emph{$Y$-$\Delta$-exchange at $D$}.

\paragraph{Seymour's decomposition.}
Before stating Seymour's decomposition theorem, we introduce two important regular matroids that are neither graphic nor cographic.
For representation matrices
\[
A^{10}:=\begin{bmatrix}1&0&0&1&1\\1&1&0&0&1\\0&1&1&0&1\\0&0&1&1&1\\1&1&1&1&1\end{bmatrix},\qquad
A^{12}:=\begin{bmatrix}1&0&1&1&0&0\\0&1&1&1&0&0\\1&0&1&0&1&1\\0&1&0&1&1&1\\1&0&1&0&1&0\\0&1&0&1&0&1\end{bmatrix} \enspace,
\]
we have that $\Rten$ can be represented over $\binary$ by $(I_5\ A^{10})$, and $\Rtwelve$ can be represented over $\binary$ by $(I_6\ A^{12})$.
\begin{theorem}[{Seymour~\cite{seymour1980decomposition}, see~\cite[Thm.~11.3.14]{truemper1992matroid}}]\label[theorem]{thm:decomp}
	Every regular matroid $M$ can be decomposed into graphic and cographic matroids and matroids isomorphic to $\Rten$ by repeated $1$-, $2$-, and $\Delta$-sum decompositions.
\end{theorem}

\section{Decomposing regular matroids with graphic leafs}\label{sec:decomposing}
The crucial structural insight on regular matroids for our algorithm is \Cref{prop:3-connected-child}.
The statement has been derived recently by B{\'e}rczi, M{\'a}trav{\"o}lgyi, and Schwarcz~\cite{berczi2024reconfiguration} to solve two conjectures for regular matroids.
Their proof is based on a global statement for the refined decomposition of regular matroids due to Aprile and Fiorini~\cite{aprile2022regular}.
We give a new short proof of \Cref{prop:3-connected-child} that is self-contained outside of text-book proofs of Seymour's decomposition~\cite{truemper1992matroid,oxley2006matroid} and relies only on local information of the decomposition.

\begin{restatable}[{B{\'e}rczi, M{\'a}trav{\"o}lgyi, and Schwarcz~\cite[Proposition~5.7]{berczi2024reconfiguration}}]{proposition}{threeconnectedchild}\label[proposition]{prop:3-connected-child}
	Let $M$ be a $3$-connected regular matroid, such that $M$ is not graphic, cographic or isomorphic to $\Rten$.
	Then there are $3$-connected regular matroids $M_1$ and $M_2$, such that $|E(M_2)|\ge9$, $M_2$ is graphic and $M_1\threesum M_2\in\{M,\dual M\}$.
\end{restatable}

Our proof of \Cref{prop:3-connected-child} builds on an inductive variant of the $3$-connected case of Seymour's decomposition, see~\cite[Lemma~11.3.18]{truemper1992matroid}.
In a nutshell, every $3$-connected regular matroid that is not isomorphic to $\Rten$ can be written as a $\Delta$-sum of a graphic or cographic matroid $M_1$ with a regular matroid $M_2$.
If $M_1$ is graphic, then our proof is done.
Otherwise, we inductively apply the variant of Seymour's decomposition to the dual matroid.
Note that the $\Delta$-sum $M_1\threesum M_2$ is not stable under dualization, since the dual of a circuit may not be a circuit, that is, the circuit that is used for identification cannot be used anymore.
Instead, the matroids $M_1$ and $M_2$ have to be modified in order to replace this circuit by a cocircuit.
This operation is called a $\Delta$-$Y$-exchange.
We complete the proof by analyzing minors of graphs and regular matroids under $\Delta$-$Y$-exchanges.
This analysis guarantees that the number of elements shrinks in each induction step.

We begin by stating an inductive variant of the $3$-connected case of Seymour's decomposition that crucially powers our proof of \Cref{prop:3-connected-child}.

\begin{lemma}[{\cite[Lemma~11.3.18]{truemper1992matroid}}]\label[lemma]{lem:inductive-decomposition}
	Let $M$ be a $3$-connected, regular matroid that is not graphic, cographic, or isomorphic to $\Rten$.
	Let $X$ be either a triangle or a single element of $M$.
	Then there exist $3$-connected regular matroids $M_1$ and $M_2$ such that $M=M_1\threesum M_2$, $M_1$ contains $X$ and $M_2$ is graphic or cographic.
	Moreover, $|E(M_2)|\ge9$.
\end{lemma}

The following statement on the dual of a $\Delta$-sum can be shown by an analysis of the respective sets of bases.
\begin{lemma}[{\cite{mcguinness2014base}}]\label[lemma]{dual}
	Let $M_1$ and $M_2$ be binary matroids on ground sets $E_1$ and $E_2$ such that $E_1\cap E_2=\{d_1,d_2,d_3\}:=D$, where, for both matroids, $D$ forms a circuit and does not contain a cocircuit. 
	We denote $M_{iY}:=Y_D(M_i)$ for $i\in[2]$.
	Then, we can relate the $\Delta$-sum at $D$ with the help of duality and $\Delta$-$Y$-exchanges at $D$ as follows
	\[M_{1}\threesum M_{2}\simeq\dual{(\dual{M_{1Y}}\threesum\dual{M_{2Y}})}.\]
\end{lemma}

We point out a direct connection between the notions of $\Delta$-$Y$-exchanges and $\Delta$-sums, that will be helpful in the proof of \Cref{prop:3-connected-child}.
\begin{lemma}[{\cite[Proposition~11.5.8]{oxley2006matroid}}]\label[lemma]{delta-wye}
	Let $M$ be a binary matroid with three specified elements $D:=\{d_1,d_2,d_3\}$ such that $D$ forms a circuit and does not contain a cocircuit.
	Label the edges of an arbitrary triangle in $K_4$ by the elements of~$D$.
	Then, 
	\[Y_D(M)\simeq M\threesum M(K_4).\]
\end{lemma}
\begin{proof}
	We denote the remaining elements of $M(K_4)$ by $y_1,y_2,y_3$, such that $y_i$ is adjacent to $d_j$ where $j\in[3]\setminus\{i\}$.
	We define a map $\phi:E(Y_D(M))\to E(M\threesum M(K_4))$ mapping any element of $E(M)\setminus D$ to itself, and any $d_i$ to $y_i$ for $i\in[3]$.
	Further, we naturally extend $\phi$ to subsets of elements.
	We claim that $\phi$ is an isomorphism between the matroids $Y_D(M)$ and $M\threesum M(K_4)$ by giving a natural mapping between the respective sets of bases.
	
	Given a basis $B$ of $Y_D(M)$, by the definition of $\Delta$-$Y$-exchanges, there is a corresponding basis $B'$ of $M$ such that $B$ and $B'$ differ only on $D$ and $|B|-|B'|=1$.
	Furthermore, there is a basis $B''$ of $M(K_4)$ such that $B''\cap\{y_1,y_2,y_3\}=\phi(B\cap D)$.
	We claim that we can choose $B'$ and $B''$ in a way such that they fulfill the axioms for the $\Delta$-sum.
	
	First, clearly $|B'\cap D|+|B''\cap D|=2$. 
	If either term is $0$, there is nothing to show.
	Otherwise, $|B\cap D|=2$ and $B'=B\symdiff D$.
	By the basis-exchange property for matroids, there is some $x_i\in B\cap D$ such that $(B'\setminus D)\cup \{x_i\}$ is a basis of~$M$.
	We define $B'':=\phi(B\cap D)\cup (B\cap D-d_i)$.
	It is easy to check that $B'$ and $B''$ fulfill all necessary properties.
	
	The converse direction uses the same construction. 
\end{proof}

\begin{lemma}\label[lemma]{planar}
	Let $G=(V,E)$ be a $2$-connected non-planar graph with a $3$-edge cut $\{e_1,e_2,e_3\}$ that partitions $V$ into vertex sets $A,B$.
	If $|A|,|B|\ge2$, then there is an $i\in[3]$, such that $G/e_i$ is non-planar.
\end{lemma}
\begin{proof}
	We prove the statement by contradiction.
	Assume there is a $2$-connected graph $G$ with either a $K_{3,3}$ or a $K_5$-model and associated partition of the vertex set $(V_1,V_2,\dots)$.
	Further, let $\{e_1,e_2,e_3\}$ be a $3$-edge cut with partition $A,B$ of the vertex set with $|A|,|B|\ge2$.
	If $e_i=v_{i_1}v_{i_2}$ such that $\{v_{i_1},v_{i_2}\}\subseteq V_j$ for some $i\in[3]$ and some $j\in[6]$, then $G/e_i$ has a model of $K_{3,3}$ or $K_5$.
	Thus, any $V_j$ is a subset of $A$ or of $B$.
	This is not possible for $K_5$.
	Therefore, we can assume that $(V_1,\dots,V_6)$ is a model of $K_{3,3}$, see~\Cref{fig:graphs}.
	
	Further, only one $V_j$ can be contained in $A$, we assume without loss of generality that $V_1=A$.
	We denote the endpoints of the edges $\{e_1,e_2,e_3\}$ by $R:=A\cap\left(\bigcup_{i\in[3]}e_i\right)$.
	Let $r\in R$.
	Since $G$ is $2$-connected, any connected component of $G[A\setminus\{r\}]$ has to contain some vertex $r'\in R$.
	Since $|A|\ge2$, this implies in particular that $|R|\ge2$, and there is some $r_1\in R$ that is incident to exactly one of $\{e_1,e_2,e_3\}$.
	If $G[A\setminus\{r_1\}]$ is not connected, then it has exactly two connected components, each containing a vertex in $R$, which we denote by $r_2$ and $r_3$. 
	Thus there is an $r_1$-$r_2$-path in $G[A]$ that does not contain $r_3$.
	Therefore, there is a vertex $r^*\in R$ that is incident to exactly one edge in $\{e_1,e_2,e_3\}$ (we assume that this is $e_1$ without loss of generality), such that $G[A\setminus\{r^*\}]$ is connected.
	Let $j\in[6]$, such that $e_1$ crosses from $V_1$ to $V_j$.
	Then, we can move $r^*$ from $V_1$ to $V_j$ and maintain that $(V_1,\dots,V_6)$ is a model of $K_{3,3}$, see~\Cref{fig:graphs}.
	But then, $e_1$ is contained in $V_j$, and therefore $G/e_1$ is non-planar.
\end{proof}

\begin{figure}[!h]
	\centering
	\includegraphics[width=4cm]{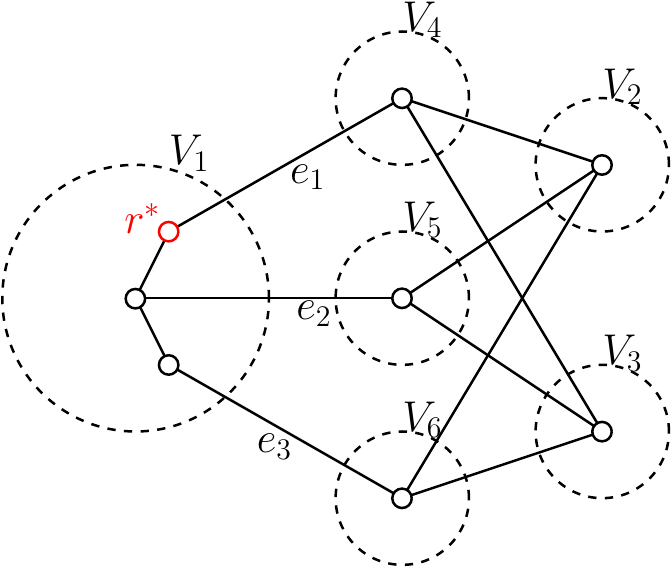}\hspace{2cm}
	\includegraphics[width=4cm]{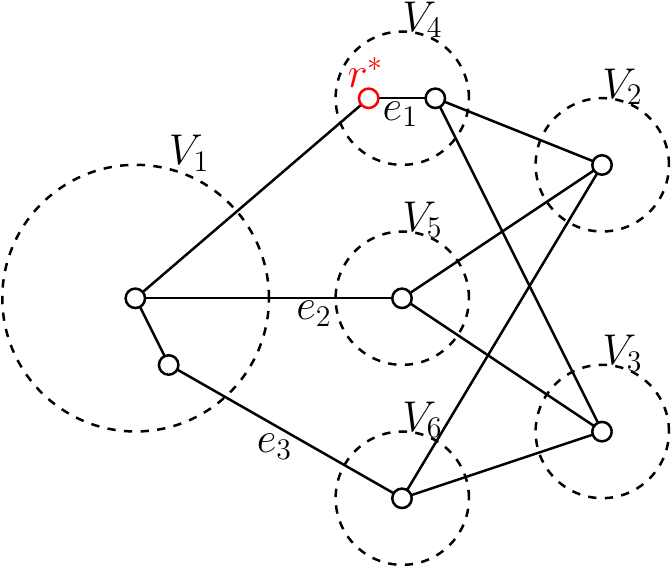}
	\caption{Modification of the partition of a $K_{3,3}$-model, such that both endpoints of the edge $e_1$ are contained in the set $V_4$. The figure on the left shows the original model, and the figure on the right shows the modified model.}
	\label{fig:graphs}
\end{figure}

\begin{lemma}\label[lemma]{dy}
	Let $M$ be a $3$-connected regular matroid that is cographic, but not graphic and let $D=\{d_1,d_2,d_3\}$ be a circuit of $M$ that does not contain a cocircuit.
	If $Y_D(M)$ is graphic, then it is also cographic.
\end{lemma}
\begin{proof}
	We show the contrapositive statement, i.e., if $Y_D(M)$ is not cographic, then it is also not graphic.
	Let $G$ be a graph such that $\dual{M(G)}=M$.
	Then $D$ corresponds to a $3$-edge cut in $G$, inducing a partition $(A,B)$ of $V(G)$.
	If $|A|=1$ or $|B|=1$, then the $\Delta$-$Y$-exchange preserves cographicness by \Cref{delta-wye}.
	
	By the definition of $\Delta$-$Y$-exchanges, we have $Y_D(M)/d_i\simeq M\backslash d_i$ for $i\in[3]$.
	Recall that the deletion of elements from a cographic matroid corresponds to contracting the corresponding edge in $G$.
	Therefore, by \Cref{planar}, there is some $i\in[3]$, such that $\dual{M(G/d_i)}=M\backslash d_i$ is a cographic matroid that is not graphic.
	Thus, $Y_D(M)$ has a non-graphic minor and therefore is non-graphic itself.
\end{proof}
We are now in position to give a short proof of \Cref{prop:3-connected-child}.
\begin{proof}[Proof of \Cref{prop:3-connected-child}]
	We prove the statement by induction on the number of elements of~$M$, extending on \Cref{lem:inductive-decomposition}.
	Clearly, $\Rtwelve$ is the unique smallest regular matroid that is $3$-connected and not graphic, cographic or isomorphic to $\Rten$. 
	Further, it has the claimed decomposition.
	Let $M$ be a $3$-connected, regular matroid that is not graphic, cographic, or isomorphic to $\Rten$ with at least $13$ elements, and let $M_1$, $M_2$ be a decomposition of $M$ provided by \Cref{lem:inductive-decomposition}, where $X$ is an arbitrary element of $M$.
	We denote the common triangle of $M_1$ and $M_2$, which is joined by the $\Delta$-sum, by $D$.
	If $M_2$ is graphic, then the decomposition fulfills the claimed properties.
	
	By \Cref{dual}, we have that $\dual M=\dual{(M_{1Y})}\threesum\dual{(M_{2Y})}$.
	Observe that connectivity is invariant under duals, so $\dual M$ is $3$-connected, as is $\dual{(M_{2Y})}$.
	Further, $D$ corresponds to a triangle of $\dual{(M_{2Y})}$.
	If $\dual{(M_{2Y})}$ is graphic, then the decomposition fulfills the claimed properties. 
	By \Cref{dy}, we have that if $\dual{(M_{2Y})}$ is cographic, then it is also graphic, in which case we are again done.
	Since $\dual{(M_{2Y})}$ contains a triangle, it cannot be isomorphic to $\Rten$. 
	Therefore, we can apply \Cref{lem:inductive-decomposition} on $\dual{(M_{2Y})}$ with $X=D$.
	Since $\dual{(M_{2Y})}$ has strictly less elements than $M$, the claim follows by induction.
\end{proof}

The following lemma shows roughly speaking that cocircuits induced by vertex cuts of graphic matroids are preserved as cocircuits under a $\Delta$-sum. 
More precisely, if $M$ is a graphic matroid, $C$ a cocircuit induced by a vertex cut of $M$, and $D$ a triangle in $M$ such that $C\cap D=\emptyset$, then $C$ is a cocircuit of $M\threesum N$, where $N$ is a regular matroid and the $\Delta$-sum is performed on~$D$.
Together with \Cref{prop:3-connected-child}, this implies that for any $3$-connected regular matroid $M$ that is not $\Rten$, there is a cocircuit in $M$ or $\dual M$ that `locally behaves' like a vertex of a graph.

\begin{restatable}{lemma}{cutcocircuit}\label[lemma]{cut_cocircuit}
	Let $M_1$ be a regular matroid, $M_2=M(G)$ be a $3$-connected graphic matroid, where $G$ has $\ell\ge4$ vertices.
	Let $M:=M_1\threesum M_2$ such that the $\Delta$-sum is performed on the triangle $T:=\{v_1v_2,v_1v_3,v_2v_3\}$ of $G$.
	Then, $\{e\in E(G):v_\ell\in e\}$ is a cocircuit of $M$.
\end{restatable}
\begin{proof}
	The definition of cocircuits implies that $D\subseteq E(M)$ is a cocircuit of $M$ if and only if it intersects every basis of $M$ and it is inclusion-wise minimal for this property.
	By the definition of the $\Delta$-sum, any basis of $M$ can be mapped to a basis of $M_2$ (potentially adding some of the triangle edges). 
	Conversely, any basis of $M_2$ can be extended to a basis of $M$.
	Hence, a set of elements $D\subseteq E(M_2)\setminus T$ is a cocircuit of $M$ if and only it is a cocircuit in $M_2$.
	Clearly, $\{e\in E(G):v_\ell\in e\}$ is a cocircuit of $M_2$ and hence a cocircuit of $M$.
\end{proof}

\section{Reduction to the 3-connected case}\label{sec:oneandtwosum}

Before we demonstrate how we can use our technical insights on regular matroids to handle the $3$-connected case, in this section, we demonstrate how to deal with $1$- and $2$- sums in our inductive proof of \Cref{main-thm}.

\begin{proposition}\label[proposition]{prop:solve1sum}
	Let $M_1$ and $M_2$ be matroids and suppose that $\bgp{M_1}$ and $\bgp{M_2}$ can be represented by $(+,\times,/)$-circuits of size $s_1$ and $s_2$, respectively. Then, the basis generating polynomial $\bgp{M}$ of $M=M_1\onesum M_2$ can be represented by a $(+,\times,/)$-circuit of size $s_1+s_2+1$.
\end{proposition}
\begin{proof}
	By the definition of the $1$-sum, we obtain $\bgp{M}=\bgp{M_1}\cdot\bgp{M_2}$. 
	Hence, we can simply combine the two circuits with one additional multiplication gate.
\end{proof}

\begin{restatable}{proposition}{solvetwosum}
	\label[proposition]{prop:solve2sum}
	Let $M_1$ and $M_2$ be matroids with $\{d\}=E(M_1)\cap E(M_2)$. Suppose that $\bgp{M_1}$, $\bgp{M_2\setminus d}$, and $\bgp{M_2/d}$ can be represented by $(+,\times,/)$-circuits of size $s_1$, $s_2^\setminus$, and $s_2^/$, respectively. 
	Then, the basis generating polynomial $\bgp{M}$ of $M=M_1\twosum M_2$ can be represented by a $(+,\times,/)$-circuit of size $s_1+s_2^\setminus + s_2^/ + 2$.
\end{restatable}

\begin{proof}
	The general idea is to group the bases of $M$ into two parts, depending on which side contains the gluing element $d$. This allows to write $f_M$ in terms of $f_{M_1}$, $f_{M_2\setminus d}$, $f_{M_2/ d}$, and two additional arithmetic operations.
	
	More precisely, we claim that the basis generating polynomial $f_M$ can be written as
	\begin{align}\label{eq:delcontr}
		f_M(x) = f_{M_1}\Big(x_{E(M_1)-d},\frac{f_{M_2\setminus d}(x_{E(M_2)-d})}{f_{M_2/d}(x_{E(M_2)-d})}\Big)\cdot f_{M_2/ d}(x_{E(M_2)-d}),
	\end{align}
	where $x_{E(M_i)-d}$ is the restriction of $x = (x_e)_{e \in E}$ to the elements of $M_i$ without~$d$ and the fraction is the argument at the position of element $d$. Before we prove~\eqref{eq:delcontr}, we quickly argue why it implies the claim. Indeed, the formula implies that we can combine $(+,\times,/)$-circuits for $f_{M_1}$, $f_{M_2\setminus d}$, and $f_{M_2/ d}$ into one for $f_M$ by adding one additional division gate and one additional multiplication gate, yielding an overall size of $s_1+s_2^\setminus + s_2^/ + 2$, as claimed.
	
	It remains to prove \eqref{eq:delcontr}. To this end we will compare the monomials appearing at both sides of the equation. The left-hand side has one monomial for each basis of $M$. Recall that by definition of the $2$-sum, these bases are
	\begin{align}
		&\phantom{{}={}}\{(B_1\cup B_2)-d:B_1\in\B_1,B_2\in\B_2,d\in B_1\symdiff B_2\}\notag\\
		&=\{(B_1-d)\cup B_2: B_1\in\B_1, B_2\in \B_2, d\in B_1\setminus B_2\}\notag\\&\phantom{{}={}}\cup \{B_1\cup (B_2-d):B_1\in\B_1, B_2\in \B_2, d\in B_2\setminus B_1\}\notag\\
		&\eqqcolon \B_{1,2}\cup\B_{2,1}.\label{eq:split}
	\end{align}
	
	Now let us look at the right-hand side of \eqref{eq:delcontr}. Looking at the first factor and partitioning it into two kinds of monomials, we obtain
	\begin{align*}
		f_{M_1}\Big(x_{E(M_1)-d},\frac{f_{M_2\setminus d}(x_{E(M_2)-d})}{f_{M_2/d}(x_{E(M_2)-d})}\Big) = \sum_{\substack{B_1 \in \B_1\\d\notin B_1}} x^{B_1} + \sum_{\substack{B_1\in\B_1\\d\in B_1}} \frac{f_{M_2\setminus d}(x_{E(M_2)-d})}{f_{M_2/d}(x_{E(M_2)-d})}\cdot x^{B_1-d}.
	\end{align*}
	
	Multiplying this with the remaining factor, we obtain that the entire right-hand side of \eqref{eq:delcontr} equals
	
	\begin{align*}
		\sum_{\substack{B_1 \in \B_1\\d\notin B_1}} f_{M_2/d}(x_{E(M_2)-d}) \cdot x^{B_1} + \sum_{\substack{B_1\in\B_1\\d\in B_1}} f_{M_2\setminus d}(x_{E(M_2)-d}) \cdot x^{B_1-d}.
	\end{align*}
	
	Plugging in the definition of deletion and contraction via bases, this further equals
	
	\begin{align*}
		\sum_{\substack{B_1 \in \B_1\\d\notin B_1}} \sum_{\substack{B_2 \in \B_2\\d\in B_2}} x^{B_1} \cdot x^{B_2-d} + \sum_{\substack{B_1\in\B_1\\d\in B_1}} 
		\sum_{\substack{B_2\in\B_2\\d\notin B_2}} x^{B_1-d}\cdot x^{B_2} = \sum_{B\in \B_{1,2}}x^B + \sum_{B\in \B_{2,1}} x^B,
	\end{align*}
	which equals $f_M(x)$ by the arguments leading to \eqref{eq:split}. This completes the proof.
\end{proof}

\section{Handling the 3-connected case}\label{sec:3-connected}
\subsection{Proof of the main result}
\begin{proposition}\label[proposition]{prop:removegraphic}
	Let $M=M_1\threesum M_2$ with $M_1$ being a regular matroid and $M_2$ being a graphic matroid that is 3-connected up to parallel elements with $n_2\coloneqq|E(M_2)|\ge9$. Assume $f_{M_1}$ can be represented by a $(+,\times,/)$-circuit of size $s_1$. Then, $f_M$ can be represented by a $(+,\times,/)$-circuit of size $s_1+n_2^3/2$. Moreover, this circuit can be constructed in polynomial time using an independence oracle of~$M$.
\end{proposition}

Once we have \Cref{prop:removegraphic}, we can prove \Cref{main-thm} by induction. We need one little additional helpful lemma to control dualization.

\begin{lemma}\label[lemma]{lem:dual}
	Let $M$ be a matroid on $n$ elements and assume $f_M$ can be represented by a $(+,\times,/)$-circuit of size $s$. Then, $f_{M^*}$ can be represented by a $(+,\times,/)$-circuit of size $s+2n$.
\end{lemma}
\begin{proof}
	Observe that $f_{M^*}(x)=x^E\cdot f_{M}((1/x_e)_{e\in E})$.
	This expression involves $n$ additional multiplications and $n$ additional divisions compared to~$f_M$, implying the statement.
\end{proof}

We restate our main theorem for convenience.
\mainthm*
\begin{proof}
	We prove the theorem by induction on the number $n$ of elements of the matroid $M$ that $f_M$ can be represented by a $(+,\times,/)$-circuit of size $g(n)\coloneqq n^3$. The base cases are when $M$ is either graphic, cographic, or isomorphic to $\Rten$. Note that this covers all regular matroids with up to 10 elements.
	
	If $M$ is graphic, the statement was proven by \cite{fomin2016subtraction}. While the authors just give a bound of $O(n^3)$, it is not difficult to verify that their construction actually ensures $n^3/2$. 
	This follows independently by applying the construction of our \Cref{prop:removevertex} to the special case of graphic matroids. 
	For the cographic case, we additionally apply \Cref{lem:dual}. This results in size at most $n^3/2 + 2n\leq n^3$, for $n\geq2$. Finally, $\Rten$ has exactly 162 bases and rank 5, so implementing $f_{\Rten}$ naively via its definition results in a circuit of size $809<1000=10^3$.
	
	From now onwards we assume that $M$ is neither graphic nor cographic nor isomorphic to~$\Rten$. 
	In each of the following cases, we will consider two matroids $M_1$ and $M_2$. We denote the number of elements of the respective ground sets by $n_1$ and $n_2$.

	If $M=M_1\onesum M_2$ for some regular matroids $M_1$ and $M_2$, then, by induction and \Cref{prop:solve1sum}, we obtain that $f_M$ can be represented by a $(+,\times,/)$-circuit of size at most
	\[
	g(n_1)+g(n_2)+1
	= n_1^3+n_2^3+1
	= (n_1+n_2)^3-3n_1n_2(n_1+n_2)+1
	< n^3.
	\]
	
	If $M=M_1\twosum M_2$ for some regular matroids $M_1$ and $M_2$ with $n_1,n_2\geq3$, assuming without loss of generality that $n_2\leq n_1$, then, by induction and \Cref{prop:solve2sum}, we obtain that $f_M$ can be represented by a $(+,\times,/)$-circuit of size at most
	\begin{align*}
		g(n_1) + 2g(n_2-1) + 2
		= n_1^3+2(n-n_1+1)^3+2
		\leq n^3,
	\end{align*}
	where we used $n=n_1+n_2-2$ and $3\leq n_2\leq n_1$ for the inequality.
	
	If none of the cases considered so far applies, then by \Cref{thm:decomp} the assumption of \Cref{prop:3-connected-child} is fulfilled. Thus, we can write either $M$ or $M^*$ as $M_1\threesum M_2$ with $M_1$ being a regular matroid and $M_2$ being a 3-connected graphic matroid with $n_2\coloneqq|E(M_2)|\ge9$. By induction, \Cref{prop:removegraphic}, and \Cref{lem:dual}, we obtain that $f_M$ can be represented by a $(+,\times,/)$-circuit of size at most
	\begin{align}\label{eq:rec}
		2n+g(n_1)+n_2^3/2 = 2n + (n-n_2+6)^3 + n_2^3/2.
	\end{align}
	Here, $n=n_1+n_2-6$. Since $M_1$ is 3-connected and the common triangle does not contain a cocircuit, $n_1\geq6$, so $n_2\leq n$. For fixed $n$ and $9\leq n_2\leq n$, it is easy to verify that this function is convex in $n_2$, so it attains its maximum at one of the boundary values $n_2=9$ or $n_2=n$. Plugging $n_2=9$ into \eqref{eq:rec} yields $2n+(n-3)^3+9^3/2$, which is at most $n^3$ for $n\geq9$. Plugging $n_2=n$ into \eqref{eq:rec} yields $2n+6^3+n^3/2$, which is also at most $n^3$ for $n\geq9$. This concludes the induction.
	
	We remark that the circuit in each of the operations of $1$-sum, $2$-sum, dualization can clearly be constructed in polynomial time. 
	Together with the polynomial bound from \Cref{prop:removegraphic}, this proves the total polynomial bound.
\end{proof}

\subsection{The star-mesh transformation}
The remaining key idea for our proof of \Cref{prop:removegraphic} is to ``remove one vertex'' from a locally graphic part of our regular matroid by performing an operation known as the \emph{star-mesh transformation}. 
For this, we first introduce the standard version in terms of weighted graphs.
The important observation for our proof is that we can apply the same construction even though $M$ is not graphic globally, by exploiting that $M_2$ is graphic.

Given a graph $G=(V,E)$ with edge weights $w(e)$ that are subtraction-free rational functions in formal edge variables, and a nonisolated vertex $r\in V$, we define the \emph{star-mesh transformation} of $G$ at $r$ as follows.
Let $N(r)$ denote the (open) neighborhood of $r$, i.e., the set of vertices adjacent to $r$ in $G$.
Let $G'$ be the graph obtained from $G$ by removing $r$ and its incident edges, and adding an edge $uv$ for $u,v\in N(r)$, merging parallel edges whenever they occur.
Further, we assign to each edge $e=vv'\in E(G')$ with $v,v'\in N(r)$ a new weight $w'(e)$ as follows:
\begin{equation*}
	w'(e)\coloneqq\begin{cases}w(e)+w(vr)w(v'r)/(\sum_{u\in N(r)}w(ru))&\text{if }\ e\in E(G)\\w(vr)w(v'r)/(\sum_{u\in N(r)}w(ru))&\text{otherwise}\end{cases}
\end{equation*}
All other edges retain their weights.
We also denote the star-mesh transformation of $(G,w)$ at $r$ by $\mathrm{SM}_r(G,w)\coloneqq(G',w')$.
The star-mesh operation allows to reduce the basis generating polynomial of $M(G)$ to the one of $M(G')$ in the following way.

\begin{theorem}[{\cite[Lemma~7.3]{fomin2016subtraction}}]
	Let $G$ be a graph with a nonisolated vertex $r\in V(G)$ and edge weights as above.
	Let $(G',w'):=\mathrm{SM}_r(G,w)$ be the star-mesh transformation of $G$ at $r$.
	Then, $f_{M(G)}(w)=f_{M(G')}(w')\cdot(\sum_{u\in N(r)}w(ru))$.
\end{theorem}

\subsection{A generalized star-mesh transformation}
In this section, we generalize the above star-mesh transformation and use it to complete our proof of \Cref{prop:removegraphic}.
We first state \Cref{prop:removevertex}, which allows us to remove a vertex from locally graphic parts of the regular matroid.
Next, we outline how to use \Cref{prop:removevertex} in order to complete the proof of \Cref{prop:removegraphic}.
The remainder of the section is devoted to the proof of \Cref{prop:removevertex}.

We note that when $M_1\threesum M_2=M$, and the $\Delta$-sum is performed on a triangle $T$ of $M_1$ and $M_2$, then the elements of $T$ are not contained in $M$, and are therefore not variables of $f_M$.
When performing the star-mesh transformation on $M_2$, we may obtain edges that are parallel to edges in $T$. We emphasize that while we merge parallel edges in the star-mesh transformation, we do not merge them with the edges of $T$.
Hence, in the following statement we distinguish between the elements in $T$ and the rest of the elements of $M_2$.

\iffalse
\begin{restatable}{proposition}{removevertex}
	\label[proposition]{prop:removevertex}
	Let $M_1$ be a regular matroid and $\ell\in\Z_{\ge4}$. 
	Let $M\coloneqq M_1\threesum^+ M(K_\ell)$ and $M'\coloneqq M_1\threesum^+ M(K_{\ell-1})$, where the respective $\Delta$-sum is performed on the triangle $\{v_1v_2,v_1v_3,v_2v_3\}$. 
	Let $z_{ij}$ be the variable of $f_M(z)$ corresponding to the edge $v_iv_j$ of $K_\ell$ and let $z'_{ij}$ be the variable of $f_{M'}(z')$ corresponding to the edge $v_iv_j$ of $K_{\ell-1}$. 
	Then, setting $z'_{ij}\coloneqq z_{ij}+(z_{i\ell}z_{j\ell})/y$ for $1\leq i < j\leq \ell-1$ with $y\coloneqq\sum_{k\in[\ell-1]}z_{k\ell}$, we have that $f_M(z)=yf_{M'}(z')$.
\end{restatable}
\fi
\begin{restatable}{proposition}{removevertex}
	\label[proposition]{prop:removevertex}
	Let $M_1$ be a regular matroid and $\ell\in\Z_{\ge4}$.
	Let $G$ be a graph on $\ell$ vertices such that $\{v_1,v_2,v_3\}\subseteq V(G)$, and assume that $v_\ell$ is nonisolated.
    Let $T:=\{v_1v_2,v_1v_3,v_2v_3\}$ be a triangle, and assume that $M_2:=M(G\cup T)$ is $3$-connected (up to parallel elements).
	Let $(G',z'):=\mathrm{SM}_{v_\ell}(G,z\vert_{E(G)})$ be the graph obtained from $G$ by performing the star-mesh operation on $v_\ell$, and extend $z'$ by $z$ on the elements of $M_1$.
	Let $M\coloneqq M_1\threesum M_2$ and $M'\coloneqq M_1\threesum M(G'\cup T)$, where the respective $\Delta$-sum is performed on the triangle $T$. 
    Finally, let $y\coloneqq\sum_{e=vv_\ell:v\in N(v_\ell)}z_{e}$.
    Then, we have that $f_M(z)=yf_{M'}(z')$.
\end{restatable}

\begin{proof}[Proof of \Cref{prop:removegraphic}]
	We show by induction on $\ell\geq3$ that $f_M$ can be represented with size $s_1 + \ell^3/2$. 
	For $\ell=3$, observe that $M_1 \threesum M(K_3\cup T)$ is isomorphic to~$M_1$. Thus, the induction start follows. Now assume that $\ell>3$. We apply \Cref{prop:removevertex} once to $M$. For computing $y$ we need at most $\ell-2$ addition gates, for computing all the $z'$-variables, we need three gates each, amounting to at most $3(\ell-1)(\ell-2)/2$ many gates, and then we need one further gate to multiply $y$ with $f_{M'}(z')$. Thus, by induction, we can compute $f_M$ with at most
	\[
	s_1 + (\ell-1)^3/2 + (\ell-2) + 3(\ell-1)(\ell-2)/2 + 1 \leq s_1 + \ell^3/2 
	\]
	many gates, finishing the induction.
    Observe that $M(G'\cup T)$ maintains $3$-connectivity up to parallel elements, as long as it has more than $3$ vertices.
\end{proof}

To derive the representation arising from the star-mesh transformation, we use the following generalization of Kirchhoffs celebrated matrix-tree theorem on the number of spanning trees of a connected graph~\cite{kirchhoff1847ueber}.
\begin{theorem}[{Maurer~\cite[Thm.~3]{maurer1976matrix}}]\label[theorem]{thm:matrix-tree}
	Let $M$ be a regular matroid and let $A\in\{0,\pm1\}^{r\times n}$ be a totally unimodular matrix representing $M$.
	Let $X:=\mathrm{diag}(x_e:e\in E(M))$ and $L:=AXA^\intercal$.
	Then, $f_M=\mathrm{det}(L)$.
\end{theorem}

\begin{definition}\label[definition]{gen-star-mesh}
	Let $A\in\{0,\pm1\}^{r\times n}$ be a matrix.
	Let $N_1:=\mathrm{supp}(A_r)$ and $N_0:=[n]\setminus N_1$, where $A_r$ denotes the $r$-th row of $A$.
	We define a matrix $A'\in\R^{(r-1)\times n'}$, where $n':=|N_0|+\binom{|N_1|}{2}$ as the matrix obtained by performing a \emph{star-mesh transformation on $A$ with respect to $r$}.
	Specifically, we index the columns of $A'$ by $N_0\cup\binom{N_1}{2}$ and define $A'_{i,j}:=A_{i,j}$ if $i\in[r-1]$, $j\in N_0$ and $A'_{i,(j,k)}:=A_{i,j}-(A_{r,j}A_{r,k}A_{i,k})$ for $i\in[r-1]$ and $j,k\in N_1$ with $j<k$.
\end{definition}

We note that our definition of star-mesh transformation for matrices generalizes the well-known star-mesh transformation for graphs (see for instance~\cite{fomin2016subtraction}) in the following sense:
If $G$ is a graph and $A\in\{0,\pm1\}^{V(G)\times E(G)}$ is the vertex-edge incidence matrix of $G$, where each edge is arbitrarily directed, then $A'$ is the (directed) vertex-edge incidence matrix of a graph $G'$, where some vertex $v\in V(G)$ has been removed along with its incident edges, and has been replaced by a complete graph on the neighborhood of $v$.
In this sense, the star-mesh transformation generalizes the notion of $Y$-$\Delta$-exchanges, with equivalence for vertices of degree $3$.

Our final ingredient to prove \Cref{prop:removevertex} is that our version of the star-mesh transformation can actually be applied to reduce the number of rows of the matrix $L$ in \Cref{thm:matrix-tree} by one. This is captured by the following proposition.

\begin{restatable}{proposition}{matrixstarmesh}\label[proposition]{matrix-star-mesh}
	Let $A\in\{0,\pm1\}^{r\times n}$ be a matrix and let $X=\mathrm{diag}(z_i:i\in[n])$ be a diagonal matrix of rational functions.
	Further, let $N_1:=\mathrm{supp}(A_r)$, $N_0:=[n]\setminus N_1$, and $y:=\sum_{i\in N_1}z_i$. Assume that $y$ is not identically zero.
	Then, $\mathrm{det}(AXA^\intercal)=y\cdot\mathrm{det}(A'X'A'^\intercal)$, where the matrix $A'\in\R^{(r-1)\times n'}$ is obtained by performing a star-mesh transformation on $A$ with respect to $r$, and $X'\in\ratfun^{n'\times n'}$ is a diagonal matrix of rational functions such that $X'_{i,i}:=z_i$ for $i\in N_0$ and $X'_{(\alpha,\beta),(\alpha,\beta)}:=z_\alpha z_\beta /y$ for $\alpha,\beta\in N_1$ with $\alpha<\beta$.
\end{restatable}
\begin{proof}
	Let $L:=AXA^\intercal$, i.e., the entries of $L$ are defined by $L_{i,j}:=\sum_{k=1}^nA_{i,k}A_{j,k}z_k$ for $i,j\in[r]$.
	In order to compute the determinant of $L$, we perform elementary row operations, such that the $r$-th column of the resulting matrix $L'$ corresponds to $(0,\dots,0,L_{r,r})^\intercal$.
	To be precise, we define 
	\begin{align*}
		L'_{i,j}:&=L_{i,j}-\frac{L_{i,r}L_{j,r}}{L_{r,r}}\\
		&=\sum_{k=1}^nA_{i,k}A_{j,k}z_k-\frac{\left(\sum_{k=1}^nA_{i,k}A_{r,k}z_k\right)\cdot\left(\sum_{k=1}^nA_{j,k}A_{r,k}z_k\right)}{\sum_{k=1}^nA_{r,k}^2z_k}\\
		&=\sum_{k=1}^nA_{i,k}A_{j,k}z_k-\frac{\sum_{\alpha=1}^n\sum_{\beta=1}^nA_{i,\alpha}A_{r,\alpha}A_{j,\beta}A_{r,\beta}z_\alpha z_\beta}{y},
	\end{align*}
	for $i\in[r-1]$ and $j\in[r]$, and $L'_{r,j}:=L_{r,j}$ for $j\in[r]$, where $y:=\sum_{i=1}^nA_{r,i}^2z_i=\sum_{i\in N_1}z_i$.
	Now, by Laplace expansion on the $r$-th column of $L'$, we obtain that $\mathrm{det}(L)=\mathrm{det}(L')=y\cdot\mathrm{det}(L'')$, where $L''$ is the restriction of $L'$ to the first $r-1$ rows and columns.
	We claim that $L'_{i,j}=(A'X'A'^\intercal)_{i,j}=:\widetilde L_{i,j}$ for $i,j\in[r-1]$.
	
	Note that both the definition of $L'_{i,j}$ and $A'$ are invariant under multiplying columns of $A$ with $-1$.
	Thus, we can assume without loss of generality that $A_{r,i}=1$ for $i\in N_1$.
	Therefore, we have that $A'_{i,(j,k)}:=A_{i,j}-A_{i,k}$ for $i\in[r-1]$ and $j,k\in N_1$ with $j<k$.
	
	We proceed to modify the expression of $L'$ in the following way:
	\begin{align*}
		L'_{i,j}&=\sum_{k\in N_0}A_{i,k}A_{j,k}z_k+\sum_{k\in N_1}\left(A_{i,k}A_{j,k}z_k-A_{i,k}A_{j,k}\frac{z_k^2}{y}\right)\\
		&\hspace{4cm}-\sum_{\substack{\alpha,\beta\in N_1\\\alpha<\beta}}(A_{i,\alpha}A_{j,\beta}+A_{j,\alpha}A_{i,\beta})\frac{z_\alpha z_\beta}{y}\\
		&=\sum_{k\in N_0}A_{i,k}A_{j,k}z_k+\sum_{k\in N_1}A_{i,k}A_{j,k}\frac{z_k(y-z_k)}{y}\\
		&\hspace{4cm}-\sum_{\substack{\alpha,\beta\in N_1\\\alpha<\beta}}(A_{i,\alpha}A_{j,\beta}+A_{j,\alpha}A_{i,\beta})\frac{z_\alpha z_\beta}{y}
	\end{align*}
	
	We focus on the central term:
	\begin{align*}
		\sum_{k\in N_1}A_{i,k}A_{j,k}\frac{z_k(y-z_k)}{y}&=\sum_{\alpha\in N_1}A_{i,\alpha}A_{j,\alpha}\left(\sum_{\substack{\beta\in N_1\\\alpha\neq\beta}}z_\beta\right)\frac{z_\alpha}{y}\\
		&=\sum_{\substack{\alpha,\beta\in N_1\\\alpha<\beta}}(A_{i,\alpha}A_{j,\alpha}+A_{i,\beta}A_{j,\beta})\frac{z_\alpha z_\beta}{y}
	\end{align*}
	
	Similarly, by definition of $\widetilde L$, we have
	\begin{align*}
		\widetilde L_{i,j}&=\sum_{k\in N_0}A_{i,k}A_{j,k}z_k+\sum_{\substack{\alpha,\beta\in N_1\\\alpha<\beta}}A'_{i,(\alpha,\beta)}A'_{j,(\alpha,\beta)}\frac{z_\alpha z_\beta}{y}\\
		&=\sum_{k\in N_0}A_{i,k}A_{j,k}z_k\\
		&\hspace{1cm}+\sum_{\substack{\alpha,\beta\in N_1\\\alpha<\beta}}\left((A_{i,\alpha}A_{j,\alpha}+A_{i,\beta}A_{j,\beta})-(A_{i,\alpha}A_{j,\beta}+A_{j,\alpha}A_{i,\beta})\right)\frac{z_\alpha z_\beta}{y}
	\end{align*}
	for $i,j\in[r-1]$.
	Since $L'_{i,j}=\widetilde L_{i,j}$ for all $i,j\in[r-1]$, the respective determinants coincide as claimed.
\end{proof}

We finally have all ingredients in place to give the proof of \Cref{prop:removevertex}.
\begin{proof}[Proof of \Cref{prop:removevertex}]
	By \Cref{cut_cocircuit}, the set of edges $\{vv_\ell:v\in N(v_\ell)\}$ corresponds to a cocircuit of $M$ and by \Cref{fund_cocircuit}, we can efficiently find a binary representation $A^{\binary}$ of $M$, such that $\{vv_\ell:v\in N(v_\ell)\}=\mathrm{supp}(A^{\binary}_r)$, where $A^{\binary}_r$ denotes the $r$-th row of $A^{\binary}$.
	Further, we can efficiently find a signing such that $A\in\{0,\pm1\}^{r\times n}$ is totally unimodular and $\mathrm{supp}(A^{\binary})=\mathrm{supp}(A)$, see~\cite{camion1964matrices} and~\cite[Corollary~9.2.7]{truemper1992matroid}.
	It remains to show that the matrix $A'$ as defined in \Cref{matrix-star-mesh} represents the matroid $M'$ up to parallel elements.
	
	In order to see this, consider the matrix $A^+$, which is obtained by adding the columns corresponding to pairwise differences of columns in the support of $A_r$, and not removing the corresponding columns.
	We denote the corresponding linear matroid by $M^+$.
	Clearly, the element corresponding to the difference of two of the columns forms a circuit together with the elements corresponding to the original columns.
	Note that the corresponding edges $vv_\ell$, $v'v_\ell$ in $G$ are adjacent, and hence the difference of the columns corresponds to the edge $vv'$.
    Hence, the corresponding matroid is regular, and the representation corresponds to the canonical representation with respect to the chosen basis, and $A^+$ is totally unimodular.
	Removing the columns in the support of $A_r$ hence corresponds to deleting $v_\ell$ from $G$.
	We remove the created parallel copies and sum the corresponding entries in $X'$, which preserves $A'X'A'^\intercal$ and coincides with the definition of $z'_{ij}$.
	
	Thus, we can apply \Cref{thm:matrix-tree} and \Cref{matrix-star-mesh} to obtain
	\begin{equation*}
		f_M=\mathrm{det}(AXA^\intercal)=y\cdot\mathrm{det}(A'X'A'^\intercal)=y\cdot f_{M'}.
	\end{equation*}
\end{proof}

\section{Generalization to MFMC matroids}\label{sec:mfmc}
While our original motivation for the problem was for regular matroids, in fact it is easy to check that our algorithm to construct arithmetic circuits for the basis generating polynomial works more broadly in matroids that can be decomposed via $1$- and $2$-sums into base blocks that admit the efficient construction of such an arithmetic circuit.
This is in line with previous results using Seymour's decomposition, see for instance~\cite{dinitz2014matroid,berczi2024reconfiguration}.
One central such class of matroids is the \emph{Max-Flow-Min-Cut (MFMC) matroids}, which were introduced by Seymour~\cite{seymour1977matroids}, and admit a similar decomposition theorem to regular matroids.

To be more precise, we define the matroid $F_7$ as the binary matroid with representation matrix $A:=(I_3\,A^7)$, with
\[
A^7:=\begin{bmatrix}1 & 0 & 1 & 1\\1 & 1 & 0 & 1\\0 & 1 & 1 & 1\end{bmatrix}.
\]
Then, the class of MFMC matroids is defined to be the class of matroids that can be decomposed into regular matroids and copies of $F_7$ by repeated $1$-, and $2$-sum decompositions.
As such, MFMC matroids form a proper minor-close subclass of binary matroids, that contains all regular matroids. 
Since $F_7$ has bounded size, and our construction for the $1$- and $2$-sum works for general matroids, see \Cref{prop:solve1sum} and \Cref{prop:solve2sum}, we obtain the following corollary.

\begin{corollary}\label[corollary]{cor:MFMC}
	For a MFMC matroid $M$ with $n$ elements, there is a $(+,\times,/)$-circuit of size $O(n^3)$ computing the basis generating polynomial $\bgp{M}$.
	Given an independence oracle of $M$, this circuit can be constructed in polynomial time.
\end{corollary}

\begin{acknowledgements}
	We thank Matthias Walter for discussions about matroid decompositions. We thank Steven Noble for clarifications on the hardness of counting bases. We thank Samuel Fiorini for initiating the group retreat in Wissant in June 2024, during which this project gained momentum.
	We thank the anonymous referees of the conference version for valuable comments that significantly improved the article.
	Part of this work was completed while Christoph Hertrich was affiliated with Université \mbox{Libre} de Bruxelles, Belgium, and received support by the European Union's Horizon Europe research and innovation program under the Marie Skłodowska-Curie grant agreement No 101153187---NeurExCo.
	Part of this work was completed while Stefan Kober was affiliated with Université \mbox{Libre} de Bruxelles, Belgium, and received support by the \emph{Fonds de la Recherche Scientifique} - FNRS through research project BD-DELTA-3 (PDR 40028812).
\end{acknowledgements}

 \section*{Conflict of interest}

 The authors declare that they have no conflict of interest.

\bibliographystyle{abbrvnat}      % mathematics and physical sciences
\bibliography{ref}   % name your BibTeX data base

\end{document}